\documentclass{article}

\usepackage{amssymb}
\usepackage{amsmath}
\usepackage{amsthm}
\usepackage{hyperref}
\newtheorem{thm}{Theorem}[section]
\newtheorem{lem}[thm]{Lemma}
\newtheorem{prop}[thm]{Proposition}
\newtheorem{cor}[thm]{Corollary}
\newtheorem{conj}[thm]{Conjecture}

\theoremstyle{definition}
\newtheorem{defn}[thm]{Definition}

\theoremstyle{remark}

\newcommand{\restrict}{\upharpoonright}
\newcommand{\la}{\langle}
\newcommand{\ra}{\rangle}

\title{On a conjecture of Dobrinen and Simpson concerning almost everywhere domination}

\author{Stephen Binns \\  Bj{\o}rn Kjos-Hanssen \\ Manuel Lerman  \\ Reed Solomon}

\begin{document}

\maketitle

\footnotetext{Solomon's research was partially supported by an NSF Grant DMS-0400754.  The authors thank Joe Miller for helpful comments 
on a draft of this article.}

\tableofcontents

\section{Introduction}
\label{sec:intro}

Dobrinen and Simpson \cite{dob:04} introduced the notions of almost everywhere domination and uniform almost everywhere domination 
to study recursion theoretic analogues of results in set theory concerning domination in generic extensions of transitive 
models of ZFC and to study regularity properties of the Lebesgue measure on $2^{\omega}$ in reverse mathematics.  
In this article, we examine one of their conjectures concerning these notions.  

Throughout this article, $\leq_T$ denotes Turing reducibility and $\mu$ denotes the Lebesgue (or ``fair coin'') probability 
measure on $2^{\omega}$ given by $\mu( \{ X \in 2^{\omega} \, | \, X(n) = i \} ) = 1/2$.  
A property holds \textbf{almost everywhere} or \textbf{for almost all} $X \in 2^{\omega}$ if it holds on a set of measure 1.   
For $f, g \in \omega^{\omega}$, $f$ \textbf{dominates} $g$ if $\exists m \forall n > m ( f(n) > g(n))$.   

\begin{defn}[Dobrinen, Simpson] 
A set $A \in 2^{\omega}$ is \textbf{almost everywhere (a.e.) dominating} if for almost all $X \in 2^{\omega}$ and all functions $g \leq_T X$, there is a
function $f \leq_T A$ such that
$f$ dominates $g$.  $A$ is \textbf{uniformly almost everywhere (u.a.e.) dominating} if there is a function $f \leq_T A$ such that for almost all $X \in
2^{\omega}$ and all functions $g \leq_T X$, $f$ dominates $g$.   
\end{defn}

There are several trivial but useful observations to make about these definitions.  First, although these properties 
are stated for sets, they are also properties of Turing degrees.  That is, a set is (u.)a.e.~dominating if and only if every other set of the 
same degree is (u.)a.e.~dominating.  Second, both properties are closed upwards in the Turing degrees.  Third, 
u.a.e.~domination implies a.e.~domination.  Finally, if $A$ is u.a.e.~dominating, then there is a function $f \leq_T A$ which 
dominates every computable function.

Dobrinen and Simpson \cite{dob:04} introduced these notions to study the following two regularity properties of $\mu$ 
in reverse mathematics: for each $G_{\delta}$ set 
$Q \subseteq 2^{\omega}$ and each $\epsilon > 0$, there is a closed set $F \subseteq Q$ such that $\mu(F) \geq \mu(Q) - \epsilon$, 
and for each $G_{\delta}$ set $Q \subseteq 2^{\omega}$, there is an $F_{\sigma}$ set $S \subseteq Q$ such that $\mu(Q) = \mu(S)$.  
$\text{ACA}_0$ is strong enough to prove these statements, so as the first step toward establishing reversals, 
they proved the following two theorems.  (Reverse mathematics plays only a motivational role here, but the reader who 
is not familiar with this subject is referred to Simpson \cite{simp:book}.)  

\begin{thm}[Dobrinen, Simpson] 
\label{thm:simpson}
For $A \in 2^{\omega}$, the following are equivalent.
\begin{enumerate}
\item $A$ is a.e.~dominating.
\item For every $\Pi^0_2$ set $Q \subseteq 2^{\omega}$ and $\epsilon > 0$, there is a $\Pi^{0,A}_1$ set $F \subseteq Q$ such that 
$\mu(F) \geq \mu(Q) - \epsilon$.  
\end{enumerate}
\end{thm}

\begin{thm}[Dobrinen, Simpson]
\label{thm:other}
For $A \in 2^{\omega}$, the following are equivalent.
\begin{enumerate}
\item $A$ is u.a.e.~dominating.
\item For every $\Pi^0_2$ set $Q \subseteq 2^{\omega}$, there is a $\Sigma^{0,A}_2$ set $S \subseteq Q$ such that 
$\mu(Q) = \mu(S)$.  
\end{enumerate}
\end{thm}

Given these connections, it is reasonable to think that results in computability theory concerning a.e.~domination and u.a.e.~domination will have 
implications for the reverse mathematics content of the regularity properties stated above.  At the time of Dobrinen and Simpson \cite{dob:04}, 
several facts about u.a.e.~domination were already known:
\[
A \geq_T 0' \, \Rightarrow \, A \, \text{is u.a.e.~dominating} \, \Rightarrow \, A' \geq_T 0''.
\]
The first implication follows from a result of Kurtz \cite{kurtz:phd} that $0'$ is u.a.e.~dominating and the second implication follows from 
a result of Martin \cite{mar:66} that $A$ computes a function which dominates every computable function if and only $A' \geq_T 0''$.   
(A set $A$ for which $A' \geq_T 0''$ is called \textbf{high}.)   Furthermore, Dobrinen and Simpson \cite{dob:04} 
presented an unpublished proof by Martin that no computable set is a.e.~dominating.  

Several questions arise naturally from these implications.  Does 
every u.a.e.~dominating set compute $0'$?  Is every high degree u.a.e.~dominating or at least a.e.~dominating?  Is every a.e.~dominating 
degree high?  Does a.e.~domination imply u.a.e.~domination?

Cholak, Greenberg and Miller \cite{cho:ta} recently answered the first question in the negative 
by showing that there is a c.e.~set $A <_T 0'$ which is u.a.e.~dominating.  They also used their methods to show a number of results in 
reverse mathematics concerning the regularity property that for every $G_{\delta}$ set $Q$ there is an $F_{\sigma}$ set $S \subseteq Q$ 
such that $\mu(Q) = \mu(S)$.  In particular, this property does not imply $\text{ACA}_0$ even over $\text{WKL}_0$.  The fourth question 
remains open.  Concerning the second and third questions, Dobrinen and Simpson made the following conjecture.  

\begin{conj}[Dobrinen, Simpson]
\label{conj:aedom}
$A' \geq_T 0'' \Leftrightarrow A$ is a.e.~dominating.
\end{conj}

This conjecture is our main focus.  The strongest results appear in Section \ref{sec:delta2} where we show that if 
$A \leq_T 0'$ is a.e.~dominating, then $A$ is high (giving a partial answer to the $\Leftarrow$ direction of Conjecture \ref{conj:aedom})
and that for any a.e.~dominating set $Z$, every set which is 1-random relative to $Z$ is actually 2-random.  As a corollary (applying 
work of Nies \cite{nies:ta}), we obtain the stronger property that every a.e.~dominating set $A \leq_T 0'$ satisfies $0'' \leq_{tt} A'$, 
where $\leq_{tt}$ denotes truth table reducibility.  (Such sets are called \textbf{super high}.)  Because there are $\Delta^0_2$ 
(even $\Sigma^0_1$) sets which are high but not super high, this result refutes the $\Rightarrow$ direction of Conjecture \ref{conj:aedom}.  

Before arriving at Section \ref{sec:delta2}, we follow a meandering path to explore 
the connections between a.e.~domination and notions such randomness and genericity.  
Because relatively little is known about a.e.~domination, we approach this property from different angles and occasionally offer more than 
one proof of our results.  Hopefully, others will see additional connections and push this work towards a more complete understanding 
of this property.  

In Section \ref{sec:ce}, we give a direct construction of a high computably enumerable (c.e.) set $H$ which is not 
a.e.~dominating.  The construction combines Martin's technique for showing the computable sets are not a.e.~dominating with a standard technique for 
constructing high c.e.~degrees.  Noam Greenberg and Joe Miller independently obtained a similar (although not c.e.) result using a different 
method.  In Section \ref{sec:random}, we show that if $A$ is 2-random, then $A$ is not a.e.~dominating and hence the 
measure of all a.e.~dominating sets is 0.  Furthermore, we show that almost every degree is bounded by a high degree which 
is not a.e.~dominating.  It follows that there are $2^{\aleph_0}$ many counterexamples to the $\Rightarrow$ direction of Conjecture \ref{conj:aedom}.  
In Section \ref{sec:generic}, we prove that if $A$ is 2-generic (with respect to Cohen forcing), then $A$ is not a.e.~dominating.  Furthermore, 
we show that for any a.e.~dominating $A$, there is a 2-random $R$ whose degree is c.e.~in $A$.    

In Section \ref{sec:ideals}, we approach Conjecture \ref{conj:aedom} from the viewpoint of Turing ideals.  Suppose that an ideal $\mathcal{I}$ 
satisfies $\forall X \in \mathcal{I} \, \exists Y \in \mathcal{I} \, (X <_T Y \wedge X'' \leq_T Y')$.  
Must $\mathcal{I}$ be a Scott set (that is, contain a path through each infinite subtree of $2^{< \omega}$ contained 
in $\mathcal{I}$)?  In other words, must such an ideal be the second order part of an $\omega$-model of $\text{WKL}_0$?   
We show that for any computable tree $T \subseteq 2^{< \omega}$ without a computable infinite path, there is a ideal closed under the 
highness property above that does not contain a path through $T$.  Hence, this ideal is 
not a Scott set and does not give an $\omega$-model of $\text{WKL}_0$.  Along 
the same lines, we show that there is an ideal closed under this highness property that does not contain an a.e.~dominating set.    

Our notation is standard and mostly follows Soare \cite{soa:book}.  $\Phi_e$ denotes the list of partial computable functionals and we fix a computable 
bijection $\langle x,y \rangle$ between $\omega^2$ and $\omega$.  For $A \subseteq \omega$, let
$A^{[e]} = \{ x \, | \, \langle e,x \rangle \in A \}$ and for $m \in \omega$, let $[m,\infty) = \{ n \, | \, n \geq m \}$.  
We sometimes equate sets with reals from the interval $[0,1]$ by viewing a set $B \subseteq \omega$ as the real 
$\sum_{n \in B} 2^{-n}$.  For strings $\sigma$ and $\tau$, we write $\sigma \sqsubseteq \tau$ to indicate that $\sigma$ is an 
initial substring of $\tau$.  Similarly, if $F_1$ is a finite set and $F_2$ is a set, then we write $F_1 \sqsubseteq F_2$ if $F_1$ is 
an initial segment of $F_2$.   

\section{Computably enumerable example}
\label{sec:ce}

In this section, we give a direct construction of a high c.e.~set which is not a.e.~dominating.  In Section \ref{sec:delta2}, we give an 
alternate proof of Theorem \ref{thm:cehigh} using index sets.

\begin{thm}
\label{thm:cehigh}
There is a high c.e.~set which is not a.e.~dominating.
\end{thm}

To prove this theorem, we build a c.e.~set $H$ such that $\emptyset'' \leq_T H'$ and $H$ is not a.e.~dominating.
By negating the property of a.e.~domination, it suffices to prove there is a set $T \subseteq 2^{\omega}$ of positive measure such that 
$\forall X \in T \, \exists g \leq_T X \, \forall f \leq_T H \, (g \, \text{is not dominated by} \, f)$.

Fix any rational number $\delta$ such that $0 < \delta < 1$.  We build a partial computable functional $\Phi$ such that
$\mu(\{ X \, | \, \Phi(X) \, \text{is not total} \}) \leq \delta$ and we let $T = \{ X \, | \, \Phi(X) \, \text{is total} \}$.  Therefore, $\mu(T) \geq 1-\delta$ as required.
Furthermore, we ensure that for every $e \in \omega$ and for every $X \in T$, if $\Phi_e(H)$ is total, then $\Phi(X)$ is not dominated by $\Phi_e(H)$.
These properties suffice to prove the theorem.

There are two types of requirements: $R_e$ requirements which guarantee that 
$H$ is high and $M_e$ requirements which define the functional $\Phi$ and the set $T$.    
Our construction takes place on a tree of strategies which is described below.  If $\alpha$ is an $R_e$ or $M_e$ strategy, then we let 
$A^{[\alpha]}$ and $H^{[\alpha]}$ denote $A^{[e]}$ and $H^{[e]}$, $\langle \alpha, x \rangle$ denote $\langle e,x \rangle$ and 
$\Phi_{\alpha}$ denote $\Phi_e$.  

To make $H$ high, we use the following standard trick.  Let $\text{Tot}$ denote the
index set of total computable functions and let $A$ be a c.e.~set such that for all $e$,
if $e \in \text{Tot}$, then $A^{[e]} = \omega$ and if $e \not \in \text{Tot}$, then $A^{[e]}$ is a finite initial segment of $\omega$.  To make $H$ high, 
it suffices to build $H \subseteq A$ such that if $A^{[e]}$ is infinite then $A^{[e]} - H^{[e]}$ is finite.  It follows that if $e \in \text{Tot}$, 
then $H^{[e]}$ is cofinite (and hence $\lim_x H(\langle e,x \rangle) = 1$) and if $e \not \in \text{Tot}$, then $H^{[e]}$ is finite (and hence 
$\lim_x H(\langle e,x \rangle) = 0$).  Because $\text{Tot} = \lim_x H(\langle e,x \rangle)$, the Limit Lemma gives 
$\emptyset'' \equiv_T \text{Tot} \leq_T H'$.

Let $R_e$ denote the requirement that $H^{[e]} \subseteq A^{[e]}$ and $A^{[e]}$ infinite implies $A^{[e]} - H^{[e]}$ is finite.  
An $R_e$ strategy $\alpha$ operates under a finite restraint and maintains a parameter $o_{\alpha}$ which 
is larger than this restraint and which is only changed when some higher priority strategy raises its restraint and initializes $\alpha$.  
$\alpha$ acts as follows.
\begin{enumerate}
\item When $\alpha$ first acts or has been initialized, define $o_{\alpha}$ to be large and set $n = o_{\alpha}$.
\item Wait for $n$ to enter $A^{[\alpha]}$.
\item When $n$ enters $A^{[\alpha]}$, put $n$ into $H^{[\alpha]}$, increase the value of $n$ by 1 and return to Step 2.
\end{enumerate}
Clearly this strategy makes $H^{[\alpha]} \subseteq A^{[\alpha]}$.  
If $\alpha$ is on the true path, then the higher priority strategies initialize $\alpha$ only 
finitely often.  Therefore, the parameter $o_{\alpha}$ reaches a limit and every number in $A^{[\alpha]}$ larger that the final value of 
$o_{\alpha}$ enters $H^{[\alpha]}$.  

An $R_e$ strategy $\alpha$ has infinitely many possible outcomes: the numbers in $\omega$ (which denote the current 
value of $o_{\alpha}$) and $\text{Fin}$.  These outcomes are ordered by 
$n <_L \text{Fin}$ for every $n \in \omega$ and $n <_L m$ if $n > m$.  (That is, $\text{Fin}$ is the rightmost outcome and the numerical outcomes 
increase in value as they move to the left.)  The strategy takes outcome $o_{\alpha}$ each time it acts in Step 3 and it takes outcome 
$\text{Fin}$ otherwise.  If $A^{[\alpha]}$ is finite, then the strategy is eventually stuck in Step 2 forever and cofinitely often
takes the $\text{Fin}$ outcome.  On the other hand, as long as $\alpha$ is initialized only finitely often, if $A^{[\alpha]}$ is infinite, then 
there is a final value of $o_{\alpha}$ for which $\alpha$ takes outcome $o_{\alpha}$ infinitely often.   

The second type of requirement concerns building the partial computable functional $\Phi$ and the set $T$.  Globally, we need to make sure that
$\mu(T) \geq 1 - \delta$, and locally we let $M_e$ denote the requirement that if $\Phi_e(H)$ is total, then $\Phi(X)$ is not dominated by $\Phi_e(H)$ for any $X \in T$.
To avoid domination, it is enough to make sure that for each $e$ and each $X \in T$, there is at least one value $x$ such that $\Phi(X;x) > \Phi_e(H;x)$.  (See Lemma 
\ref{lem:Mmet} for a proof that this condition is sufficient.)  The action for a single $M_e$ strategy $\alpha$ proceeds as follows.
\begin{enumerate}
\item Pick a small value $\epsilon_{\alpha} = 1/2^{p_{\alpha}}$ for some large $p_{\alpha}$.  
(We discuss below how to choose this number, but in particular $\epsilon_{\alpha} < \delta$.)
\item Divide $2^{\omega}$ into $2^{p_{\alpha}}$ many disjoint clopen sets $U_1^{\alpha}, \ldots, U_{2^{p_{\alpha}}}^{\alpha}$ each of size $\epsilon_{\alpha}$.
\item Cycle through the $U_i^{\alpha}$ sets beginning with $i = 1$.
\begin{enumerate}
\item Pick a large value $x_i^{\alpha}$ and define $\Phi(X;x_i^{\alpha}) = 0$ for all $X \not \in U_i^{\alpha}$.
\item Wait for $\Phi_{\alpha}(H;x_i^{\alpha})$ to converge.
\item If $\Phi_{\alpha}(H;x_i^{\alpha})$ converges, then define $\Phi(X;x_i^{\alpha}) > \Phi_{\alpha}(H;x_i^{\alpha})$ for all $X \in U_i^{\alpha}$, 
increase $i$ by 1 and return to Step 3(a).  To preserve the
computation $\Phi_{\alpha}(H;x_i^{\alpha})$, restrain $H$ from changing below the use of this computation.
\end{enumerate}
\item If $i$ eventually runs through Step 3 for all the numbers $1, \ldots, 2^{p_{\alpha}}$ then stop the action for $\alpha$ and declare it satisfied.
\end{enumerate}

Consider what such a strategy does in isolation.  If $\Phi_{\alpha}(H)$ is total, then it runs through the cycle in Step 3 for each $i$ between $1$ and $2^p$ and defines
$\Phi$ such that $\forall X \, \exists x \, (\Phi(X;x) > \Phi_{\alpha}(H;x))$.  This action wins $M_e$.  
If $\Phi_{\alpha}(H)$ is not total, then there may be an $i$ between $1$ and $2^{p_{\alpha}}$ for which
$\Phi(H;x_i^{\alpha})$ does not halt.  In this case, $\alpha$ gets stuck in Step 3(b) during the $i^{\text{th}}$ cycle.  
Consequently, $\Phi(X;x_i^{\alpha})$ does not converge for any $X \in U_i^{\alpha}$ and so $\Phi(X)$ is not total for any $X \in U_i^{\alpha}$.  
However, because $\alpha$ does not progress past Step 3(b) of the $i^{\text{th}}$ cycle,
the sets $X \in U_i^{\alpha}$ are the only sets on which $\alpha$ causes $\Phi$ to be partial.  Therefore, $M_e$ is won trivially (since
$\Phi_{\alpha}(H)$ is not total) and $\alpha$ contributes 
a set of measure $\epsilon_{\alpha}$ on which $\Phi$ is not total.   In each of the two
cases described here, $\alpha$ imposes only finitely much restraint since it has only finitely many cycles to run through.

The action of $\alpha$ when $\Phi_{\alpha}(H)$ is not total tells us how to pick the values of $\epsilon$ during the construction.  Each strategy which acts for
an $M_e$ requirement potentially contributes a set of measure $\epsilon$ on which $\Phi$ is not total.  Therefore, we have to choose the values of
$\epsilon$ as we go through the construction so that the sum of these values (over all $e$) is less than $\delta$.

The outcomes for an $M_e$ strategy $\alpha$ are the numbers in $\omega$ (indicating the current restraint imposed by $\alpha$).  
These outcomes are ordered by $n <_L m$ if $n > m$.  (That is, the numerical values of the outcomes increase as they move left.)  

We define the tree of strategies by induction.  The empty string $\lambda$ is assigned to $R_0$.  If $\alpha$ is assigned to $R_e$,
then $\alpha*\text{Fin}$ and $\alpha*n$ (for $n \in \omega$) are assigned to $M_e$.  If $\alpha$ is assigned to $M_e$,
then $\alpha*n$ (for $n \in \omega$) is assigned to $R_{e+1}$.  In both cases, the outcomes are ordered as described above.   

How do the strategies interact?  Suppose $\alpha$ is an $M$ strategy and $\beta$ is an $R$ strategy.  If $\alpha$ is to the left of $\beta$, 
then $\beta$ is initialized whenever $\alpha$ acts.  In particular, when $\beta$ nexts acts, it picks a new value for $o_{\beta}$ 
which is larger than the restraint (if any) imposed by $\alpha$.  If $\alpha$ is an $M$ strategy such that $\alpha*m \sqsubseteq \beta$, then 
$\beta$ is not eligible to act until $\alpha$ has imposed restraint up to $m$.  Therefore, $o_{\beta}$ is chosen $> m$ and $\beta$ respects 
$\alpha$'s restraint.  On the other hand, if $\beta*\text{Fin} \sqsubseteq \alpha$, then each time $\beta$ puts a number into $H^{[\beta]}$, $\alpha$ 
is initialized.  If $\beta*\text{Fin}$ is on the true path, then eventually $\beta$ always takes outcome $\beta*\text{Fin}$.  Therefore, $\beta$ 
causes only finitely much injury to $\alpha$.  

The one nontrivial interaction is when $\beta*m \sqsubseteq \alpha$.  In this case, $\alpha$ is only eligible to act if $\beta$ sets $o_{\beta}$ to be $m$, 
and $\alpha$ guesses that $\beta$ will eventually put every number greater than $o_{\beta} = m$ into $H^{[\beta]}$.  Therefore, when $\alpha$ 
sees a convergent computation $\Phi_{\alpha}(H;x_i^{\alpha})$ with use $u$, it only believes the computation if every number $\langle \beta,x \rangle \leq u$ 
with $o_{\beta} = m \leq x$ is in $H$.  Because $\alpha$ believes that $\beta$ will place every such $x$ into $H^{[\beta]}$, $\alpha$ believes that any 
computation missing such a number will eventually be destroyed by the enumeration of $x$ into $H^{[\beta]}$.  Therefore, the general construction 
contains this minor modification for an $M$ strategy $\beta$.  

We now present the formal construction.  At stage 0, let $H_0 = \emptyset$.  At stage $s > 0$, we let strategies act beginning with the $R_0$ strategy $\lambda$
until we reach a strategy of length $s$.  Once we reach a strategy of length $s$, end the stage and initialize all strategies of lower priority than the last strategy
eligible to act.  Initializing an $R$ strategy $\alpha$ means canceling $o_{\alpha}$ and $n_{\alpha}$.  Initializing an $M$ strategy $\alpha$ means 
canceling $r_{\alpha}$ (the current restraint imposed by $\alpha$), $\epsilon_{\alpha}$ and $p_{\alpha}$, canceling the partition $U_i^{\alpha}$ and 
canceling all witnesses $x_i^{\alpha}$.  Any parameters not canceled by initialization retain their values at the next stage.
Once the initialization is done, we define $\Phi(X;y) = 0$ for all sets $X$ and all numbers $y \leq s$ which are not currently witnesses of the form 
$x_i^{\alpha}$ for some $M$ strategy $\alpha$.  (Formally, we choose a large value $k$ and let $\Phi(\sigma,y) = 0$ for all strings $\sigma$ of length $k$.)

When an $R$ strategy $\alpha$ is eligible to act, it acts as follows.  If $s$ is the first stage at which $\alpha$ is eligible to act or if $\alpha$ has been
initialized since it was last eligible to act, define $o_{\alpha}$ to be large and set $n_{\alpha} = o_{\alpha}$.  Check if $n_{\alpha} \in A^{[\alpha]}_s$.  
(We begin at this step if $n_{\alpha}$ is already defined.)  If not, then let $\alpha*\text{Fin}$ be the next strategy
eligible to act.  If so, then enumerate $n_{\alpha}$ into $H^{[\alpha]}$, increase $n_{\alpha}$ by 1 and let
$\alpha*o_{\alpha}$ be the next strategy eligible to act.

When an $M$ strategy $\alpha$ is eligible to act, it acts as follows.  If $s$ is the first stage at which $\alpha$ is eligible to act or if $\alpha$
has been initialized since it was last eligible to act then we need to define $r_{\alpha}$, $\epsilon_{\alpha}$ and $p_{\alpha}$.  Set $r_{\alpha} = 0$.  
(The parameter $r_{\alpha}$ denotes $\alpha$'s current level of restraint.)  Let $q$ be the sum of all $\epsilon_{\gamma}$ parameters defined by all
$M$ strategies $\gamma$ that have been eligible to act at any time during the construction so far.  In the
verification below, we prove that $q < \delta$.  Define $p_{\alpha}$ to be a large number so that $\epsilon_{\alpha} = 1/2^{p_{\alpha}}$ satisfies
$q + \epsilon_{\alpha} < \delta$.  Partition $2^{\omega}$ into $2^{p_{\alpha}}$ many disjoint clopen sets $U^{\alpha}_1, \ldots, U^{\alpha}_{2^{p_{\alpha}}}$
each of size $\epsilon_{\alpha}$.

Begin the cycles for $\alpha$ with this choice of $p_{\alpha}$.  (If $p_{\alpha}$ was already defined, we start the action of $\alpha$ wherever
it left off in this cycle procedure.)  Run cycles beginning with $i=1$ and proceeding through $i = 2^{p_{\alpha}}$.  The $i^{\text{th}}$ cycle acts as follows.
Pick a large value for the witness $x^{\alpha}_i$ when the cycle begins and define $\Phi(X;x^{\alpha}_i) = 0$ for all $X \not \in U^{\alpha}_i$.
Let $\beta_0, \ldots, \beta_{k-1}$ denote the $R$ strategies such that $\beta_j*m_j \sqsubseteq \alpha$ for some $m_j \in \omega$.  Check if 
$\Phi_{\alpha,s}(H_s;x_i^{\alpha})$ converges.  If not, let $\alpha*r_{\alpha}$ be the next strategy eligible to act.  If it does converge, then let $u$ 
be the use of the computation.  For each $0 \leq j \leq k-1$, check if for every number $\langle \beta_j,y \rangle \leq u$ with $m_j \leq y$, we have 
$y \in H_s^{[\beta_j]}$.  If not, then let $\alpha*r_{\alpha}$ be eligible to act (and $\alpha$ remains in the $i^{\text{th}}$ cycle 
when it is next eligible to act).  If so, then define $\Phi(X;x_i^{\alpha}) > \Phi_{\alpha}(H_s;x_i^{\alpha})$ for all $X \in U_i^{\alpha}$.  Redefine $r_{\alpha}$ 
to be the maximum of its old value and $u$, increase $i$ by 1 (so that $\alpha$ will begin the $(i+1)^{\text{st}}$ cycle when it is next eligible 
to act) and let $\alpha*r_{\alpha}$ be eligible to act.  If $i_{\alpha}$ reaches the value $2^{p_{\alpha}}+1$, then $\alpha$ performs no 
further actions (unless it is initialized) and takes outcome $\alpha*r_{\alpha}$ at all future stages.  

This completes the description of the formal construction.  A strategy $\alpha$ is on the true path 
if $\alpha$ is the leftmost strategy of length $|\alpha|$ which is eligible to act infinitely often.  
A stage at which $\alpha$ is eligible to act is called an $\alpha$ stage.  

\begin{lem}
$H \subseteq A$.
\end{lem}

\begin{proof}
Numbers are enumerated into $H$ only after they have entered $A$.  
\end{proof}

\begin{lem}
\label{lem:highmet}
If $\alpha$ is an $R$ strategy on the true path, then $o_{\alpha}$ reaches a limit $\hat{o}_{\alpha}$.  Furthermore, if $\alpha*\hat{o}_{\alpha}$ 
is on the true path, then $A^{[\alpha]} = \omega$ and $A^{[\alpha]} - H^{[\alpha]}$ is finite, while if $\alpha*\text{Fin}$ is on the true path, then 
$A^{[\alpha]}$ is finite.  
\end{lem}

\begin{proof}
Let $s$ be the first $\alpha$ stage such that no strategy to the left of $\alpha$ is eligible to act after stage $s$ and hence $\alpha$ is never initialized after stage $s$.  
Because values of $o_{\alpha}$ are canceled only by initialization, $\alpha$ defines the final value for $o_{\alpha}$ at stage $s$.  

Once $o_{\alpha}$ has reached its limit $\hat{o}_{\alpha}$, there are only two possible outcomes for $\alpha$ to take at any future $\alpha$ 
stage: $\hat{o}_{\alpha}$ and $\text{Fin}$.  Because $\hat{o}_{\alpha}$ is greater than the restraint imposed on $\alpha$ by any $M$ strategy of 
higher priority, $\alpha$ is free to place any number bigger than $\hat{o}_{\alpha}$ which enters $A^{[\alpha]}$ into $H^{[\alpha]}$.  
Recall that $A^{[\alpha]}$ is either $\omega$ or a finite initial segment of $\omega$, that $\alpha$ places $n_{\alpha}$ into $H^{[\alpha]}$ and 
increments $n_{\alpha}$ (beginning at $\hat{o}_{\alpha}$) whenever it sees $n_{\alpha}$ enter $A^{[\alpha]}$, and that $\alpha$ only takes outcome 
$\alpha*\hat{o}_{\alpha}$ when it puts the current value of $n_{\alpha}$ into $H^{[\alpha]}$.  

Assume that $\alpha*\hat{o}_{\alpha}$ is eligible to 
act infinitely often.  This situation implies the interval $I = [\hat{o}_{\alpha},\infty)$ is contained in $A^{[\alpha]}$ and hence $A^{[\alpha]} = \omega$.  
Furthermore, each element of $I$ is placed into $H^{[\alpha]}$ and hence $A^{[\alpha]} - H^{[\alpha]}$ is finite.  
On the other hand, assume that $\alpha*\text{Fin}$ is on the true path.  In this case, there must be a value of $n_{\alpha}$ for which $\alpha$ never 
sees $n_{\alpha}$ enter $A^{[\alpha]}$ and hence $A^{[\alpha]}$ is a finite initial segment of $\omega$.  
\end{proof}

It follows from the previous two lemmas that each requirement $R_e$ is met by the construction, so $H$ has high degree.   

\begin{lem}
\label{lem:sum}
Let $\alpha$ be an $M$ strategy that is eligible to act at stage $s$.  Let $q_s^{\alpha}$ be the sum of all $\epsilon_{\gamma}$ parameters for all $M$ strategies
$\gamma$ that have been eligible to act at any point in the construction before $\alpha$ is eligible to act at stage $s$.  Then $q_s^{\alpha} < \delta$.
\end{lem}

\begin{proof}
This lemma follows by induction on $s$, and for each $s$ by a subinduction on the strategies which are eligible to act at stage $s$.  
If $\alpha$ defines $\epsilon_{\alpha}$ at stage $s$, then by induction $q_s^{\alpha} < \delta$, so $\alpha$ can 
define $\epsilon_{\alpha}$ such that that $q_s^{\alpha} + \epsilon_{\alpha} < \delta$.    
\end{proof}

Let $T = \{ X \, | \, \Phi(X) \, \text{is total} \}$ and let $q_s = q_s^{\alpha}$ where $\alpha$ is the last $M$ strategy eligible to act at stage $s$.  By Lemma 
\ref{lem:sum}, $q_s < \delta$ and hence $\lim_s q_s \leq \delta$.  In other words, the sum of all parameters $\epsilon_{\alpha}$ chosen by $M$ strategies 
during the construction is $\leq \delta$.  

\begin{lem}
$\mu(T) \geq 1 - \delta$.
\end{lem}

\begin{proof}
At the end of each stage $s$, if $x \leq s$ and $x$ is not equal to the current value of an $x_i^{\alpha}$ parameter for some $M$ strategy 
$\alpha$, then we define $\Phi(X;x) = 0$ for all $X$.  Fix a number $y$ and calculate $\mu( \{ X \, | \, \Phi(X;y) \uparrow \})$.  If $\Phi(X;y) \uparrow$, 
then for each stage $s \geq y$, $y$ must be equal to the current value of some $x_i^{\alpha}$ parameter.  Because new values for these parameters 
are always chosen large (and are never reused), there must be a fixed $M$ strategy $\alpha$ and a fixed cycle number $i$ such that 
$y = x_i^{\alpha}$ at all stages $s \geq y$.  (In particular, $\alpha$ is never initialized after stage $y$.)  
When $\alpha$ chose $x_i^{\alpha} = y$, it defined $\Phi(Y;y) = 0$ for all $Y \not \in U_i^{\alpha}$.  Therefore, if $\Phi(X;y) \uparrow$, then 
$X \in U_i^{\alpha}$.  Because $\mu(U_i^{\alpha}) = \epsilon_{\alpha}$, we have that either $\Phi(X;y)$ converges for all $X$ or 
$\mu(\{ X \, | \, \Phi(X;y) \uparrow \}) = \epsilon_{\alpha}$.  Summing over all $y$, we see that 
$\mu(\{ X \, | \, \Phi(X) \, \text{not total} \, \})$ is bounded by the sum of all values of $\epsilon_{\beta}$ chosen over the course of the construction 
by all $M$ strategies $\beta$.  As noted above, this sum is $\leq \delta$.
\end{proof}

\begin{lem}
\label{lem:total}
If $\alpha$ is an $M$ strategy on the true path such that $\Phi_{\alpha}(H)$ is total, then
\[
\forall X \, \exists x \, (\Phi(X;x) > \Phi_{\alpha}(H;x)).
\]
\end{lem}

\begin{proof}
Assume $\alpha$ is on the true path and $\Phi_{\alpha}(H)$ is total.  Let $\beta_0, \ldots, \beta_{k-1}$ be the $R$ strategies 
such that $\beta_j*m_j \sqsubseteq \alpha$ for some $m_j \in \omega$.  Because each 
$\beta_j*m_j$ is on the true path, if $\hat{o}_{\beta_j}$ denotes the final value of the parameter ${o}_{\beta_j}$, then 
$\hat{o}_{\beta_j} = m_j$.  By the proof of Lemma \ref{lem:highmet}, the interval 
$[m_j,\infty)$ is contained in $H^{[\beta_j]}$ for each $0 \leq j \leq k-1$.  

Let $s$ be the first $\alpha$ stage after which $\alpha$ is never 
initialized.  At stage $s$, $\alpha$ defines the parameters $p_{\alpha}$ and $\epsilon_{\alpha}$, sets $r_{\alpha}=0$ and defines the 
partition $U_i^{\alpha}$.  After these definitions, $\alpha$ begins its first cycle for defining $\Phi$.  It chooses $x_1^{\alpha}$ and defines 
$\Phi(X, x_1^{\alpha}) = 0$ for all $X \not \in U_1^{\alpha}$.   
At each $\alpha$ stage $t \geq s$, $\alpha$ checks whether $\Phi_{\alpha,t}(H_t;x_1^{\alpha})$ converges and if so 
whether each number $\langle \beta_j,y \rangle$ below the use with 
$m_j \leq y$ is in $H_t$.  Because $\Phi_{\alpha}(H)$ is total and each interval $[m_j,\infty) \subseteq H^{[\beta_j]}$, 
$\alpha$ must eventually see a convergent computation which meets this criterion.  When $\alpha$ sees an appropriate 
computation at stage $t \geq s$, it defines $\Phi(X;x_1^{\alpha}) > \Phi_{\alpha,t}(H_t;x_1^{\alpha})$ and redefines its restraint $r_{\alpha}$ to be 
greater than the use of $\Phi_{\alpha,t}(H_t;x_1^{\alpha})$.  

Since $t \geq s$, no strategy to the left of $\alpha$ is ever eligible to act after $t$, so none of these strategies can place a number 
into $H$ which will destroy the $\Phi_{\alpha,t}(H_t;x_1^{\alpha})$ computation.  The $\beta_j$ strategies have already placed all the 
numbers below the use into $H^{[\beta_j]}$ so they will not destroy this computation.  Any $R$ strategy $\beta$ with 
$\beta*\text{Fin} \sqsubseteq \alpha$ never places any more elements into $H$ since if it did, the path would move to the left of 
$\alpha$ contradicting the fact that $t \geq s$.  Finally, all strategies of lower priority that $\alpha$ respect $\alpha$'s new restraint.  
Therefore, $\Phi_{\alpha}(H;x_1^{\alpha}) = \Phi_{\alpha,t}(H_t;x_1^{\alpha})$ and we have met the condition of this lemma for 
all $X \in U_1^{\alpha}$.

We repeat the same argument for $\alpha$'s remaining cycles to see that for each $U_i^{\alpha}$, there is a witness 
$x_i^{\alpha}$ such that $\Phi(X;x_i^{\alpha}) > \Phi_{\alpha}(H;x_i^{\alpha})$ for all $X \in U_i^{\alpha}$.  Since the $U_i^{\alpha}$ 
partition $2^{\omega}$, we have established the lemma.
\end{proof}

\begin{lem}
\label{lem:Mmet}
Each requirement $M_e$ is met.
\end{lem}

\begin{proof}
Assume that some requirement $M_e$ is not met.  $M_e$ is not met means that $\Phi_e(H)$ is
total and that for some set $X \in T$, $\Phi(X)$ is dominated by $\Phi_e(H)$.  Fix $n$ such that for all $x > n$, $\Phi(X;x) < \Phi_e(H;x)$.
Let $e'$ be an index for a partial computable functional such that for all sets $Z$ and all numbers $x$, if $x < n$, then
$\Phi_{e'}(Z;x) = \Phi(X;x)+1$ and if $n \leq x$, then $\Phi_{e'}(Z;x) = \Phi_e(Z;x)$.  (Since $X \in T$, the computations
$\Phi(X;x)$ for $x < n$ are defined, so we are just fixing the same finite initial segment of $\Phi_{e'}(Z)$ for every $Z$.)
Let $\alpha'$ be the $M_{e'}$ strategy on the true path.  We have that $\Phi_{\alpha'}(H)$ is total and for all $x$,
$\Phi(X;x) < \Phi_{\alpha'}(H;x)$.  These facts directly contradict Lemma \ref{lem:total}.
\end{proof}

\section{Random examples}
\label{sec:random}

In this section, we show that almost every degree is not a.e.~dominating and that 
almost every degree is bounded by a high degree that is not a.e.~dominating.   
In contrast with the c.e.~set $H$ of the last section, all of the examples here satisfy some degree of 
randomness and hence have DNR (diagonally nonrecursive) degree.

\begin{defn}
A \textbf{Martin-L\"of test} relative to a set $A$ is a sequence $\la U_n:n\in\omega\ra$ of 
$\Sigma^{0,A}_1$ classes which is uniform in $A$ such that $\mu(U_n) \leq 2^{-n}$ for each $n$.  
$R$ is \textbf{$n$-$A$-random ($n$-random relative to $A$)} if for each
Martin-L\"of test relative to $A^{(n-1)}$ (the $n-1^{\text{st}}$ jump of $A$), we have $R\not\in
\bigcap_n U_n$. If $A$ is computable we say that $R$ is $\mathbf{n}$\textbf{-random}.  
$R$ is \textbf{weakly $n$-$A$-random} if for each $\Sigma^{0,A}_n$ class $C$ of
measure 1, we have $R \in C$.
\end{defn}

Notice two consequences of these definitions: if $R$ is $n$-$A$-random then $R$ is weakly $n$-$A$-random and if 
$R$ is $n$-$A$-random for some $A$ then $R$ is $n$-random.  For more information 
about randomness (including various equivalent definitions), see Kautz \cite{kautz:phd}, 
Kurtz \cite{kurtz:phd} or the online manuscript of Downey and Hirschfeldt \cite{dow:online}.  One of the 
fundamental results about randomness that we will use repeatedly is the following.  

\begin{thm}[Martin-L\"{o}f \cite{ml:66}]
\label{thm:fund}  
For any $A$ and $n \geq 1$, the measure of the class of all $n$-$A$-random sets is 1.
\end{thm}

We now state the main result of this section which we will prove at the end of this section.  

\begin{thm} 
\label{2}
Each 4-random degree is bounded by a high 2-random degree that is not almost everywhere dominating. 
\end{thm}

\begin{cor}
Almost every degree is bounded by a high degree which is not a.e.~dominating.  
\end{cor}

Towards this theorem, as we would like to find degrees that are not almost everywhere dominating,
we need examples of functions that are hard to dominate but are nevertheless computable by a sufficiently (to be specified later) random
oracle.  That is, suppose we fix functions $f_R \leq_T R$ for each sufficiently random set $R$.  By Theorem \ref{thm:fund}, the measure 
of such $R$ is 1.  Let $A$ be any a.e.~dominating set and let $S$ be a class of sets of measure 1 such that every function computable 
from an element of $S$ is dominated by some function computable from $A$.  Because $S$ has measure 1 and the collection of 
sufficiently random $R$ has measure 1, some such $R$ must be in $S$.  Therefore, 
the a.e.~dominating set $A$ must compute a function $g$ which dominates some $f_R$ function.  
If we can make the $f_R$ functions hard to dominate, we can use them to construct examples of sets which 
are not a.e.~dominating.  We begin with the following theorem.  (Kurtz \cite{kurtz:phd} proved that the class of sets $R$ such that 
there is a set $B <_T R$ for which $R$ is c.e.~in $B$ has measure 1 and Kautz \cite{kautz:phd} later strengthen this result to 
Theorem \ref{REA}.)  

\begin{thm}[Kautz \cite{kautz:phd}]
\label{REA}
If $R$ is 2-random, then there is a set $B$ such that $B <_T R$ and $R$ is c.e.~in $B$.  
\end{thm}

We will combine this theorem with the following simple observation.  Let $R$ be any set and suppose that $R$ 
is c.e.~in $B$.  For any fixed index $e$ such that $R = W_e^B$, we can define the computation function 
$c$ for $R$ relative to this index $e$ by 
\[
c(x)=\mu s(W^B_{e,s}\restrict x=W^B_e\restrict x = R \restrict x).  
\]
Typically, we will abuse notation by suppressing the index $e$ and referring to $c$ as ``the'' computation function for $R$ as a 
c.e.~set in $B$.

\begin{lem}
\label{comp}
If $R$ is c.e.~in $B$ and $f$ dominates the computation function for $R$ as a c.e.~set in $B$, then
$f \oplus B \geq_T R$.
\end{lem}

\begin{proof} 
Assume that $e$ is the index relative to which the computation function is defined.  Because $f$ dominates the 
computation function, for sufficiently large $x$ we have $R \restrict x = W_e^B \restrict x = W_{e,f(x)}^B \restrict x$.
\end{proof}

For any 2-random set $R$, fix $B_R$ and $f_R$ such that $B_R <_T R$, $R$ is c.e.~in $B_R$ and $f_R$ is the 
computation function for $R$ as a set c.e.~in $B_R$.  Since $f_R \leq_T B_R \oplus R$ and $B_R <_T R$, we have 
$f_R \leq_T R$.  Therefore, any a.e.~dominating set $A$ must be able to compute a function $g$ which dominates some 
$f_R$.  Hence, for some 2-random $R$, we must have $A \oplus B_R \geq_T R$.  In other words, any a.e.~dominating set 
must join some predecessor of some 2-random $R$ above $R$.  Stillwell \cite{still} showed that sufficiently random sets do not 
have this property.

\begin{lem}[Stillwell \cite{still}]
\label{still}
For any $X$,$Y$,$G$, if $X\not\le_T Y$ and $G$ is weakly 2-$X\oplus Y$-random, we have $X\not\le_T
Y\oplus G$.
\end{lem}

From Lemma \ref{still} and the comments above, it follows that if $G$ is 2-random relative to every  
2-random set $R$, then $G$ is not a.e.~dominating.  Unfortunately, there is no such set $G$.  (Suppose there is 
such a $G$.  Let $R = G$ and notice that $R$ is 2 random, but $G$ cannot be 2-random 
relative to $R = G$.)  However, van Lambalgen \cite{lamb} showed that a set $X$
can be random relative to every set that is random relative to $X$, and this turns out to be enough
to prove Theorem \ref{2randnotAED}.

\begin{thm}[van Lambalgen \cite{lamb}]
\label{lamb}
Let $n\ge 1$. If $A$ is $n$-random relative to $B$, and $B$ is $n$-random, then $B$ is $n$-random
relative to $A$ and $A\oplus B$ is $n$-random.
\end{thm}

\begin{thm} 
\label{2randnotAED}
If $G$ is 2-random, then $G$ is not a.e.~dominating.
\end{thm}

\begin{proof}
Suppose $G$ is a.e.~dominating. There is a set $S$ of measure $1$ such that for
all partial computable functionals $\Phi$ and all $X \in S$, if $\Phi(X)$ is total then $\Phi(X)$ is dominated by a function
recursive in $G$. Because the collection of all sets which are 2-random relative to $G$ has measure 1 (by Theorem \ref{thm:fund}) and because 
$S$ has measure 1, $S$ must contain some $X$ that is 2-random relative to $G$.  
In particular, $X$ is 2-random, so by Theorem \ref{REA}, there is a set 
$B$ such that $B <_T X$ and $X$ is c.e.~in $B$. Let $\Phi$ be such that $\Phi(X)$ is the 
computation function for $X$ as a set c.e.~in $B$. As $X$ is in $S$, $G$ computes a function
dominating $\Phi(X)$. By Lemma \ref{comp}, $G \oplus B \geq_T X$. By Lemma \ref{still}, $G$ is not
weakly 2-$X\oplus B$-random.  However, $X \oplus B \equiv_T X$ since $B <_T X$, so $G$ is not weakly 2-$X$-random 
and hence $G$ is not 2-$X$-random.  

$X$ is 2-random relative to $G$ and that $G$ is not 2-random relative to $X$.  Suppose for a 
contradiction that $G$ is 2-random.  By Theorem \ref{lamb}, $G$ is 2-random and $X$ is 2-random relative to $G$ 
implies that $G$ is 2-random relative to $X$, giving the desired contradiction.  
\end{proof}

\begin{cor}
\label{cor:aa}
Almost every set is not a.e.~dominating.
\end{cor}

\begin{proof}
This corollary follows from Theorems \ref{thm:fund} and \ref{2randnotAED}. 
\end{proof}

Given Theorem \ref{2randnotAED}, we can ask whether a 2-random set can be high. Kautz
\cite{kautz:phd} showed that 3-random sets cannot be high, in fact $R^{(n)}\equiv_T R\oplus
0^{(n)}$ holds for each $n+1$-random set $R$, $n\ge 1$. Also, the argument of Theorem
\ref{2randnotAED} does not generalize to all 1-random degrees, as $\mathbf{0'}$ is a 1-random degree
which is a.e.~dominating. Nevertheless we get a positive answer.   

\begin{defn}
A Turing machine $U$ is called \textbf{prefix-free} if for all finite strings $\sigma$, $U(\sigma) \downarrow$ implies that 
$U(\tau) \uparrow$ for all proper extensions $\tau$ of $\sigma$.  For any universal prefix-free Turing machine $U$, 
the \textbf{halting probability} of $U$ is 
\[
\Omega_U = \sum_{U(\sigma) \downarrow} 2^{-|\sigma|}.
\]
\end{defn}

This notion relativizes to any oracle $X$ and the following lemma 
lists three properties of $\Omega_U^X$ which will be useful for us later.  The first two properties are due to 
Chaitin \cite{cha:book} and the third is due to Kurtz \cite{kurtz:phd}.  (For more information about $\Omega$ 
numbers see \cite{cha:book} and \cite{DHMN}.)  A real $R$ is called c.e.~in $A$ if the set of rational numbers 
$q < R$ is c.e.~in $A$.  

\begin{lem}
\label{lem:omegaprop}
The following properties hold for any universal prefix-free Turing machine $U$ and any set $X$.  
\begin{enumerate}
\item $\Omega_U^X$ is a c.e.~in $X$ real.
\item $\Omega_U^X$ is 1-random relative to $X$.
\item $\Omega_U^X \oplus X \equiv_T X'$.  
\end{enumerate}
\end{lem} 

\begin{thm} 
\label{high2rand}
There is a high 2-random set below $0''$.
\end{thm}
\begin{proof}
Fix a universal prefix-free Turing machine $U$.  For any set $A$, let $R=\Omega_U^A$.  By Lemma \ref{lem:omegaprop}, 
$R$ is 1-random relative to $A$ and satisfies $A'\equiv_T R\oplus A$. Now let $A=0'$. Then $R$ is
2-random and $0''\equiv_T R \oplus 0'\le_T R'$.
\end{proof}

\begin{proof}[Proof of Theorem \ref{2}]
Let $R_1$ be 4-random. Let $R_0$ be a 2-random set with $0'' \leq_T R_0'$ (which exists by Theorem
\ref{high2rand}) and let $R=R_0\oplus R_1$. We claim that $R$ is the set we are looking for.
Clearly $R\ge_T R_1$. As $R\ge_T R_0$, $R$ is high. Since $R_0\le_T 0''$, $R_1$ is 2-random
relative to $R_0$, so by Theorem \ref{lamb}, $R$ is 2-random and hence by Theorem
\ref{2randnotAED}, $R$ is not a.e.~dominating.
\end{proof}

\section{Generic examples}
\label{sec:generic}

In this section, we show that every 2-generic degree is not a.e.~dominating.    

\begin{defn}
Let $V$ be any universal Turing machine, and let $g$ be a
computable function. \textbf{The time-bounded Kolmogorov complexity with time
bound} $\mathbf{g}$ is the function $C^g$ given by 
\[
C^g(x) = \min \{ |p| : V(p)=x \, \text{in} \, g(|x|) \, \text{steps} \, \}.
\]
(If there is no such $p$, then $C^g(x) = \infty$.)  $Z$ is \textbf{Kolmogorov random with time bound} $\mathbf{g}$ 
if there is a constant $b$ such that 
\[
(\exists^\infty n) [ C^g(Z \restrict n) \geq n-b ].
\] 
(We count $\infty > n$ for all $n \in \omega$, so the relation $C^g(Z \restrict n) \geq n-b$ is computable in $Z$.)
\end{defn}

\begin{thm}[Nies, Stephan and Terwijn \cite{nst}]
\label{nst28}
For each computable function $g$ with $g(n)\geq n^2+ O(1)$ and each set $Z$, the following are equivalent.  
\begin{enumerate}
\item $Z$ is 2-random.
\item $Z$ is Kolmogorov random with time bound $g$.
\end{enumerate}
\end{thm}

\begin{defn}
Let $A$ and $B$ be sets. We say that $A$ is \textbf{hyperimmune-free relative to} $\mathbf{B}$, denoted by $A$ is HIF($B$), if
for each function $f \leq_T A$ there is a function $g \leq_T B$ such that $f$ is dominated by $g$.
\end{defn}

The next proposition is a variation on Proposition 2.15 in \cite{nst}.  

\begin{prop}
\label{surprise}
Let $A$ be a set.  If there is a 2-random set $Z$ such that $Z$ is HIF($A$), then there is
a nonempty $\Pi^{0,A}_1$ class consisting entirely of 2-random sets.
\end{prop}

\begin{proof}
By Theorem \ref{nst28}, $Z$ is Kolmogorov random with some time bound $g$ and constant $b$. Let
$$f(m)=f^Z_{g,b}(m)=\mu n(\exists p_0,\ldots,p_m\leq n)(\forall i\leq m)[C^g(Z\restrict p_i)\geq p_i-b].$$
Note that $f \leq_T Z$ is a total function. Hence there exists $h \leq_T A$ such that $h$ dominates
$f$.  In fact, we can assume that $h$ majorizes $f$.  (That is, $h(n) \geq f(n)$ for all $n$.)  Consider the $A$-recursive tree
$$T=\{\sigma:(\forall m)[|\sigma|\geq h(m)\to(\exists
p_0,\ldots,p_m\leq|\sigma|)(\forall i\leq m)[C^g(\sigma\restrict p_i)\geq p_i-b]]\}.$$ Since $h$
majorizes $f$, $Z$ is a path on $T$ and so the set of paths of $T$ is nonempty. Moreover, each
path is time-bounded Kolmogorov random and hence 2-random by Theorem \ref{nst28}.  Therefore, the set of 
paths through $T$ is our desired $\Pi^{0,A}_1$ class.  
\end{proof}

\begin{thm}[Jockusch and Soare \cite{joc:72}, relativized]
\label{ce}
Let $A$ be any set. Each nonempty $\Pi^{0,A}_1$ class $P$ has a member
$R$ whose degree is c.e.~in $A$.
\end{thm}

\begin{proof}
The $\Pi^{0,A}_1$ class $P$ can be represented as the set of infinite paths through an $A$-computable tree $T_P \subseteq 2^{< \omega}$.  
Let $R$ be the leftmost infinite path in $T_P$ and we show that the degree of $R$ is c.e.~in $A$.  Consider the set $N$ of all finite binary strings  
which are either on $R$ or to the left of $R$ in $2^{< \omega}$.  Because $R$ is the leftmost path in $T_P$, $N$ is c.e.~in $A$, and clearly we have 
that $N \leq_T R$.  To see that $R \leq_T N$, notice that $\sigma \in N$ is an initial segment of $R$ if and only if it is the 
rightmost node of length $|\sigma|$ in $N$.  
\end{proof}

Alternately, we can view the elements of the $\Pi^{0,A}_1$ class $P$ as reals, in which case the proof of Theorem \ref{ce} says that $R$ is a 
c.e.~in A real contained in $P$.  (The set $N$ represents the rational numbers $q < R$.)  This perspective will be useful 
later when we want to view such a real as $\Omega_U^A$ for some universal prefix-free Turing machine $U$.  

\begin{prop}
\label{wow}
Let $A$ be a set.  If there is a 2-random set $Z$ such that $Z$ is HIF($A$), then there is a 2-random $R$ whose 
degree is c.e.~in $A$.
\end{prop}

\begin{proof}
By Proposition \ref{surprise}, there is a nonempty $\Pi^{0,A}_1$
class $P$ consisting of 2-random sets. By Theorem \ref{ce}, there is a
path $R$ in $P$ whose degree is c.e.~in $A$.
\end{proof}

As above, if we view the elements of $P$ as reals, then Proposition \ref{wow} says that $R$ is a c.e.~in A real which is 
2-random.   

\begin{thm}
\label{main}
If $A$ is a.e. dominating then there is a 2-random $R$ whose degree is c.e.~in 
$A$.
\end{thm}

\begin{proof}
Suppose $A$ is a.e. dominating. Let $C$ denote the class of all sets $Z$ such that every $f \leq_T Z$ is dominated by some 
$g \leq_T A$ and let $D$ denote the class of all 2-random sets.  By definition, every $Z \in C$ is HIF($A$) and because 
$A$ is a.e.~dominating, the measure of $C$ is 1.  Furthermore, since $D$ has measure 1 (by Theorem \ref{thm:fund}), 
the intersection $C \cap D$ is nonempty.  Therefore, there is a 2-random $Z$ which is HIF($A$) and 
we can apply Proposition \ref{wow}.
\end{proof}

In fact, Theorem \ref{main} also follows from Theorem \ref{feb12} below by considering $\Omega^A$. Such a proof avoids the notion of time bounded Kolmogorov complexity.   
However, this approach does not give the stronger result of Proposition \ref{wow}. 

\begin{defn}
Let $A$ and $B$ be sets.  We say that $A$ is \textbf{diagonally nonrecursive in} B, denoted by DNR($B$), if there is a function 
$f \leq_T A$ such that for all $e$, $f(e) \neq \Phi_e(B;e)$.  
\end{defn}

Notice that no set $A$ can compute a function which is DNR($A$) and that under this definition the DNR($A$) 
degrees are closed upwards trivially.   (This definition is not the only way to relativize the property of diagonally 
nonrecursiveness.)  
The following lemma is a relativized version of the result of Ku\v{c}era \cite{kuc:85} that every 1-random $R$ is DNR($\emptyset$). 

\begin{lem}
\label{lem:randomdnr}
For any $A$ and any 1-A-random $R$, $R$ is DNR($A$).
\end{lem}

\begin{proof}
Define a partition of $\omega$ by $I_0 = \{ 0 \}$, $I_1 = \{ 1,2 \}$, $I_3 = \{ 3,4,5 \}$, $\ldots$, so that $|I_n| = n+1$.  Let $f \leq_T R$ 
be defined by $f(n) = R \restrict I_n$.  (That is, $f(n)$ is the canonical index for the finite set $R \restrict I_n$.)  Let $U_k = \{ X \, | \, \exists s \exists e \geq k \, ( X \restrict I_e = \Phi_{e,s}(A;e)) \}$.
Because $|I_e|=e+1$, the measure of all sets $X$ for which $X \restrict I_e = \Phi_e(A;e)$ is at most $2^{-(e+1)}$.  
Therefore, $\mu(U_k) \le \sum_{e \geq k} 2^{-(e+1)} = 2^{-k}$. So   
the $U_k$ classes form a Martin-L\"{o}f test relative to $A$. Since $R$ is 1-$A$-random, there is a $k$ such that $R\not\in U_k$. Hence there are only finitely many $e$ for which $f(e) = \Phi_e(A;e)$, and so $R$ computes a DNR($A$) function.
\end{proof}

\begin{cor}
\label{cor:dnr}
Every 2-random $R$ is DNR$(0')$.
\end{cor}

\begin{proof}
If $R$ is 2-random, then $R$ is 1-$0'$-random.  By Lemma \ref{lem:randomdnr}, $R$ is DNR($0'$).
\end{proof}

We will apply these results to Cohen generic sets.  The forcing partial order for Cohen generics is $2^{< \omega}$ ordered by 
$\tau \leq \sigma$ (or $\tau$ is an extension of $\sigma$) if $\sigma \sqsubseteq \tau$.  

\begin{defn}
A set $G$ is called $\mathbf{n}$-$\mathbf{A}$-\textbf{generic} for forcing with a partial order $P$ if for each $\Sigma^{0,A}_n$ set 
$D \subseteq P$, $A$ either meets $D$ or $A$ meets the set of conditions in $P$ having no extension in $D$.   A subset 
$D \subset P$ is called \textbf{dense} if every $p \in P$ is extended by some $d \in D$.  
$G$ is \textbf{weakly} $\mathbf{n}$-$\mathbf{A}$-\textbf{generic} if for each $\Sigma^{0,A}_n$ dense set $D \subseteq P$, 
$G$ meets $D$.
\end{defn}

\begin{lem}
\label{lem:notdnr}
If $G$ is 2-generic, then $G'$ is not DNR$(0')$.
\end{lem}

\begin{proof}
For a contradiction, assume that $G'$ is DNR($0'$).  Each 2-generic is 1-generic and hence $G' \equiv_T G \oplus 0'$ (see for example \cite{ler:book}). So we can fix an index $i$ 
such that $\Phi_i(G \oplus 0')$ is total and for all $e$, 
$\Phi_i(G \oplus 0';e) \neq \Phi_e(0';e)$.  Consider the $\Sigma^0_2$ set 
$S = \{ \sigma \in 2^{< \omega} \, | \, \exists e, \tau, s \, ( \sigma \sqsubseteq \tau \wedge \Phi_{i,s}(\tau \oplus 0';e) 
\downarrow = \Phi_{e,s}(0';e) \downarrow) \}$.
(Whenever we deal with computations such as $\Phi_{i,s}(\tau \oplus 0';e) \downarrow$ in which the oracle has a finite component, 
we assume that the computation does not query any number in a finite component of the oracle which is larger than the length of that component.)  
Because $G$ is 2-generic and by choice of $i$, there must be an initial segment $\rho$ of $G$ such that no extension of $\rho$ is 
an element of $S$.  However, since $\Phi_i(G \oplus 0')$ is total, we know that for every $e$, there is some 
$\sigma \sqsupseteq \rho$ such that $\Phi_i(\sigma \oplus 0';e)$ converges.  Notice that $0'$ can find such a $\sigma$ by 
searching.  Furthermore, since any such $\sigma$ is not in $S$, we know that $\Phi_i(\sigma \oplus 0';e)$ is not equal to 
$\Phi_e(0';e)$.  Therefore, $0'$ can compute a function that is DNR($0'$) giving the desired contradiction.  
\end{proof}

\begin{thm}
\label{thm:2gen}
If $G$ is 2-generic then $G$ is not a.e.~dominating.  
\end{thm}

\begin{proof}
Suppose that $G$ is 2-generic and a.e.~dominating.  Because $G$ is a.e.~dominating, Theorem \ref{main} implies that there is 
a 2-random $R$ whose degree is c.e.~in $G$.  Therefore, $R \leq_T G'$.  On the other hand, by Corollary \ref{cor:dnr}, 
$R$ is DNR($0'$).  Because the DNR($0'$) degrees are closed upwards, $G'$ is DNR($0'$) which contradicts 
Lemma \ref{lem:notdnr}.  
\end{proof}

\section{Degrees below $\mathbf{0'}$}
\label{sec:delta2}

In this section, we give two proofs that that every a.e.~dominating set below $0'$ is high.  The first proof builds on 
Theorem \ref{main} while the second proof uses the notion of being ``low for random'' to establish the stronger result 
that every a.e.~dominating set is super-high.  We begin with the following lemma which states that any real 
which satisfies the first two properties of $\Omega_U^X$ in Lemma \ref{lem:omegaprop} is actually an $\Omega$ number 
for some prefix-free universal machine relative to $X$.  

\begin{lem}[Downey, Hirschfeldt, Miller and Nies \cite{DHMN}]
\label{omega}
For any set $A$ and real $R$, the following are equivalent:
\begin{enumerate}
\item $R$ is a c.e.~real relative to $A$ and 1-random
relative to $A$;
\item $R=\Omega^A_U$ for some universal prefix-free Turing machine $U$.
\end{enumerate}
\end{lem}

\begin{lem}
\label{man}
If $A$ is a.e.~dominating and $A\leq_T 0'$ then there exist
universal prefix-free machines $U$, $V$ with
$\Omega^A_U=\Omega^{0'}_V$.
\end{lem}

\begin{proof}
By Theorem \ref{main} there exists a real $R$ that is 2-random 
(that is, 1-random in $0'$) and is c.e. in $A$. Since $A\leq_T 0'$, $R$
is also 1-random in $A$ and c.e. in $0'$. Hence by Lemma
\ref{omega}, there exist $U$, $V$ such that
$R=\Omega^A_U=\Omega^{0'}_V$.
\end{proof}

\begin{thm}
\label{thm:delta2}
If $A$ is a.e.~dominating and $A \leq 0'$, then $A' \equiv_T 0''$.
\end{thm}
\begin{proof}
By Lemma \ref{man} and Property 3 of Lemma \ref{lem:omegaprop}, $A'\geq_T
\Omega^A_U \oplus 0'=\Omega^{0'}_V \oplus 0' \equiv_T 0''$.
\end{proof}

This implication can be strengthened using the following theorem from Kautz \cite{kautz:phd}.  (This theorem is a relativized form of 
a result first proved by Ku\v{c}era \cite{kuc:85}.)  For any string $\sigma$ and set $A$, let $\sigma*A$ denote the set whose characteristic 
function is $\chi(n) = \sigma(n)$ for $n < |\sigma|$ and $\chi(n) = A(n-|\sigma|)$ if $n \geq |\sigma|$.  

\begin{thm}[Kautz \cite{kautz:phd}]
\label{kuce}
Let $Z$ be a set and let $C$ be a $\Pi^{0,Z}_1$ class of positive measure.  For every 1-$Z$-random $R$, there is a string 
$\sigma$ and a set $A \in C$ such  that $R = \sigma*A$.  
\end{thm}

\begin{thm}
\label{feb12}
If $Z$ is a.e. dominating then each 1-$Z$-random is 2-random.
\end{thm}

\begin{proof}
Let $P=\overline{U_1^{0'}}=2^\omega-U_1^{0'}$ where $U_n^{0'}$, $n \in \omega$, is a
universal Martin-L\"of test relative to $0'$. Note that $P$ is a
$\Pi^0_2$ class of positive measure consisting entirely of
2-random reals. Suppose $Z$ is a.e. dominating. By Theorem
\ref{thm:simpson}, $P$ has a $\Pi^{0,Z}_1$ subclass $C$ of positive
measure. 

Let $R$ be 1-$Z$-random.  By Theorem \ref{kuce}, there is a string $\sigma$ and an $A \in C$ 
such that $R = \sigma*A$. $A$ is 2-random because it is in $C$ and we claim that $R$ is 2-random.  
For a contradiction, suppose that $R$ is not 2-random.  
Fix a $0'$ Martin-L\"{o}f test $V_n$, $n \in \omega$, such that $R \in \bigcap_n V_n$.  Let $\hat{V}_n = \{ \tau \, | \, 
\sigma*\tau \in V_{n+|\sigma|} \}$.  Because $2^{-|\sigma|} \mu(\hat{V}_n) \leq \mu(V_{n+|\sigma|})$, we have 
$\mu(\hat{V}_n) \leq 2^{|\sigma|} 2^{-(n+|\sigma|)}$ and hence $\mu(\hat{V}_n) \leq 2^{-n}$.  Therefore, 
$\hat{V}_n$ is also a $0'$ Martin-L\"{o}f test and $A \in \bigcap_n \hat{V}_n$ contradicting 
the fact that $A$ is 2-random.
\end{proof}

Because 2-random is the same as 1-$0'$-random, we can restate Theorem \ref{feb12} by saying that if $Z$ is 
a.e.~dominating, then every 
1-$Z$-random set is 1-$0'$-random.  This characterization fits the following definition from Nies \cite{nies:ta}.

\begin{defn}
\textbf{Low-for-random reducibility} $\leq_{LR}$ is defined by
$A \leq_{LR} B$ iff every 1-$B$-random set is 1-$A$-random.
\end{defn}

We can now restate Theorem \ref{feb12} as $0' \leq_{LR} Z$ for every a.e.~dominating $Z$.  Notice that if $A \leq_T B$, then 
$A \leq_{LR} B$ because every Martin-L\"{o}f test relative to $A$ is also a Martin-L\"{o}f test relative to $B$.  Nies \cite{nies:ta} 
also proved the following property of $LR$ reducibility.  (See Theorem 8.1 and the remarks before Proposition 8.3 in Nies 
\cite{nies:ta}.)  In this statement, $\leq_{tt}$ denotes truth table reducibility.  

\begin{thm}[Nies \cite{nies:ta}]
\label{thm:tt}
$A \oplus B \leq_{LR} B \, \Rightarrow A' \leq_{tt} B'$.  
\end{thm}

\begin{defn}
If $A' \geq_{tt}0''$ then we say that $A$ is \textbf{super-high}.
\end{defn}

\begin{cor}
\label{soup}
If $Z \leq_T 0'$ is a.e.~dominating then $Z' \equiv_{tt} 0''$ and hence $Z$ is super-high.
\end{cor}

\begin{proof}
Let $Z \leq_T 0'$ be a.e.~dominating.  Since $Z\le_T 0'$, we have $Z'\le_{tt}0''$ (in fact even $Z'\le_1 0''$, see for example \cite{soa:book}).  On the other hand, since $Z$ is a.e.~dominating, $0' \leq_{LR} Z$ by Theorem \ref{feb12}.  
Combining this reduction with $0' \oplus Z \leq_T 0'$ gives $0' \oplus Z \leq_{LR} Z$.  By Theorem \ref{thm:tt}, $0'' \leq_{tt} Z'$, so 
we conclude $0'' \equiv_{tt} Z'$.  
\end{proof}

Using Corollary \ref{soup}, we can give an alternate proof for Theorem \ref{thm:cehigh} using index sets.  Let 
$H_T = \{ x \, | \, 0'' \leq_T W_x' \}$ be the index set for high c.e.~sets and $H_{tt} = \{ x \, | \, 0'' \leq_{tt} W_x' \}$ be 
the index set for super-high c.e.~sets.  A proof of the following theorem can be found in Soare \cite{soa:book}.  

\begin{thm}[Schwarz \cite{schwarz}]
\label{thm:T}
$H_T$ is $\Sigma^0_5$ complete.
\end{thm}

\begin{lem}
\label{lem:tt}
$H_{tt}$ is a $\Sigma^0_4$ set.
\end{lem}

\begin{proof}
Let $\sigma_n$, $n \in \omega$ be a list of the well-formed formulas of sentential logic with sentential letters
$\mathbf{A_n}$, $n \geq 1$. Let $B$ be a set and let $v$ be a truth assignment such that $v(\mathbf{A_n})=T$ (true) iff
$n\in B$. Let $\overline v$ be the extension of $v$ to all well-formed formulas. Write $B \models \sigma_n$ if
$\overline v(\sigma_n)=T$. Then $A \leq_{tt} B$ iff there is a computable function $f$ such that for all $x$, $x \in
A$ iff $B \models \sigma_{f(x)}$. Hence 
$0'' \leq_{tt} W_x' \, \Leftrightarrow \, (\exists e)(\forall x)R(e,x)$ where
$$R(e,x) \, \Leftrightarrow \Phi_e(x)\downarrow\And [x\in 0''\iff W_x'\models\sigma_{\Phi_e(x)}].$$
So $H_{tt}$ is $\Sigma^0_2$ in $0''$, or in other words it is $\Sigma^0_4$.
\end{proof}

\begin{cor}
\label{superhigh}
There exists a high, not super-high c.e.~set.
\end{cor}

\begin{proof}
By Theorem \ref{thm:T}, $H_T$ is not a $\Sigma^0_4$ set, so
$H_{tt} \neq H_T$. As clearly $H_{tt}\subseteq H_T$, we conclude
that $H_{tt} \subsetneq H_T$. Let $x \in H_T - H_{tt}$; then $W_x$ is
high but not super-high.
\end{proof}

\begin{cor}
There is a high c.e.~set which is not a.e. dominating.
\end{cor}
\begin{proof}
Immediate from Corollaries \ref{soup} and \ref{superhigh}.
\end{proof}

\section{High-above ideals}
\label{sec:ideals}

\begin{defn}
An \textbf{ideal} is a set $\mathcal{I} \subseteq 2^{\omega}$ such that if $X \in \mathcal{I}$ and $Y \leq_T X$, then $Y \in \mathcal{I}$ and 
if $X, Y \in \mathcal{I}$, then $X \oplus Y \in \mathcal{I}$.  An ideal $\mathcal{I}$ is called a \textbf{high-above ideal} if 
\[
(\forall A \in \mathcal{I}) ( \exists B \in \mathcal{I}) ( A <_T B \, \text{and} \, 
A'' \leq_T B').
\]
\end{defn}

A trivial example of a high-above ideal is $\{ A \, | \, (\exists n)( A \leq_T 0^{(n)}) \}$.  
In this section, we construct two more examples of high-above ideals.  In Proposition 
\ref{prop:nocomp}, we use Mathias forcing to show that for any infinite computable tree 
$T \subseteq 2^{< \omega}$ with no computable path, there is a high-above ideal which does not contain a path through $T$.  Such an ideal 
is not the second order part of an $\omega$-model of 
$\text{WKL}_0$.  (To see a very different application of Mathias forcing 
in recursion theory, the reader is referred to Cholak, Jockusch and Slaman \cite{cho:01}.)  In Proposition \ref{prop:noae}, we use the 
fact that 2-random sets are not a.e.~dominating to construct a high-above ideal which does not contain an a.e.~dominating set.

We begin with a relativized version of Martin's characterization of high degrees in terms of 
dominating functions.  The proof in Chapter XI of Soare \cite{soa:book} relativizes to give the following theorem.  

\begin{thm}[Martin \cite{mar:66}, relativized]
\label{marel}
For any sets $A$, $B$, we have $(A \oplus B)' \geq_T B''$ iff 
there is a single function computable in $A\oplus B$ which dominates all functions computable in $B$.
\end{thm}

\begin{defn}
Let $H$ be any set. An $\mathbf{H}$\textbf{-computable Mathias condition} is a pair
$P=(F,C)$ where $F$ is a finite subset of $\omega$ and $C$ is an infinite $H$-computable set with
$\max(F)<\min(C)$. We say that $P_1$ \textbf{extends} $P_2$ if $F_2 \sqsubseteq F_1\subseteq F_2\cup C_2$ and  
$C_1\subseteq C_2$.  We say that a set $G$ \textbf{extends} a condition $P$ if $F \sqsubseteq G$ 
and $G \subseteq F \cup C$.
\end{defn}

We view $H$-computable Mathias conditions 
as pairs $(e,i)$ where $e$ is a canonical index for the finite set $F$ and $i$ is an index such that $C = W_i^H$.  Using this 
notation, the set of $H$-computable Mathias conditions is $\Sigma_3^H$.  Furthermore, if $(F_1, C_1)$ and $(F_2, C_2)$ 
are conditions, then the statement that $(F_1,C_1)$ extends $(F_2,C_2)$ is $\Pi_2^H$.  Therefore, when discussing 
$H$-computable Mathias forcing, we will not discuss objects which are less that 3-generic, since merely describing the forcing 
conditions and their relationships requires statements which are $\Sigma_3^H$.

\begin{lem}
\label{lem:tree}
Let $T$ be an infinite computable subtree of $2^{< \omega}$ and let $A$ be a set such that $A$ does not compute any path through $T$.  If $G$ is 
3-$A$-generic for $A$-computable Mathias forcing, then $G \oplus A$ does not compute a path through $T$.
\end{lem}

\begin{proof}
We begin by defining what it means for a Mathias condition $(F,C)$ to force various statements.  We say 
$(F,C) \Vdash \Phi_e(G \oplus A;n) \downarrow$ if $\exists s (\Phi_{e,s}(F \oplus A;n) \downarrow)$ and  
\[
(F,C) \Vdash \Phi_e(G \oplus A;n) \uparrow \, \Leftrightarrow \, \forall \, \text{finite} \, \hat{F} \, \forall s \, ( F \sqsubseteq \hat{F} \subseteq F \cup C 
\rightarrow \Phi_{e,s}(\hat{F} \oplus A;n) \uparrow ).
\]
The offset statement is equivalent to saying that no extension of $(F,C)$ forces $\Phi_e(G \oplus A;n) \downarrow$.  
Given a condition $(F,C)$, this statement is $\Pi^A_1$ and the statement that says $(F,C)$ forces $\Phi_e(G \oplus A;n) \uparrow$ 
for some $n$ is $\Sigma_2^A$.  Therefore, the set $S_e$ of all conditions $(F,C)$ for which 
$\exists n [ (F,C) \Vdash \Phi_e(G \oplus A;n) \uparrow ]$ is $\Sigma_3^{A}$.    

Assume that $G$ is 3-$A$-generic for $A$-computable Mathias forcing and that $\Phi_e(G \oplus A)$ is total.  
Because the set $S_e$ of conditions defined above is $\Sigma_3^A$, there must be a condition $(F,C)$ such that $G$ extends 
$(F,C)$ and $(F,C)$ has no extension in $S_e$.  We say that such a condition forces $\Phi_e(G \oplus A)$ to be total.  

There are two important features of conditions $(F,C)$ which force $\Phi_e(G \oplus A)$ to be total.  First, for every 
$n$ and every $(F',C')$ extending $(F,C)$, there is a condition $(F'',C'')$ extending $(F',C')$ which forces $\Phi_e(G \oplus A;n)$ 
to converge.  Second, we can take the condition $(F'',C'')$ to be a finite modification of $(F',C')$.  That is, we can add a finite number of 
elements of $C'$ to $F'$ to get $F''$ and subtract a finite number of elements from $C'$ to get $C''$.  In particular, if $(F,C)$ 
forces $\Phi_e(G \oplus A)$ to be total, then there is a set $\hat{G} \leq_T A$ for which $\Phi_e(\hat{G} \oplus A)$ is total.  
We construct $\hat{G}$ by starting with $(F_0,C_0) = (F,C)$ and choosing conditions $(F_n,C_n)$ such that $(F_{n+1},C_{n+1})$ is a finite modification 
of $(F_n,C_n)$ which extends $(F_n,C_n)$ and which forces $\Phi_e(G \oplus A;n)$ to converge.  These choices can be 
made using only the oracle $A$ since $A$ can compute $C$ and $A$ allows us to search for convergent computations of the 
form $\Phi_e(\hat{F} \oplus A;n)$ for finite extensions $\hat{F}$ of $F_n$.  The set $\hat{G} = \cup_n F_n$ clearly satisfies 
$\Phi_e(\hat{G} \oplus A)$ is total.  

Next, we consider conditions which force $\Phi_e(G \oplus A)$ to not compute a path in $T$.  Let $[T]$ denote the set of paths 
in $T$.  We say $(F,C) \Vdash \Phi_e(G \oplus A) \not \in [T] \, \Leftrightarrow$ 
\[
\exists n \, ( (F,C) \Vdash \Phi_e(G \oplus A;n) \uparrow) \, \vee \, 
\exists n \, (\Phi_e(F \oplus A) \restrict n \downarrow \wedge \Phi_e(F \oplus A) \restrict n \not \in T) ).
\]
That is, $(F,C)$ forces $\Phi_e(G \oplus A) \not \in [T]$ if $(F,C)$ either forces that $\Phi_e(G \oplus A)$ is not total or it forces 
that some initial segment of $\Phi_e(G \oplus A)$ converges to a string not in $T$.  As above, we want to say that 
$(F,C) \Vdash \Phi_e(G \oplus A) \in [T]$ if there is no extension of $(F,C)$ which forces $\Phi_e(G \oplus A) \not \in [T]$.  
If $(F,C)$ already forces $\Phi_e(G \oplus A)$ to be total, then we can write this condition as 
$(F,C) \Vdash \Phi_e(G \oplus A) \in [T] \, \Leftrightarrow$ 
\[
\forall \, \text{finite} \, \hat{F} \, \forall n \, ( F \sqsubseteq \hat{F} \subseteq F \cup C 
\rightarrow ( \Phi_{e,s}(\hat{F} \oplus A) \restrict n \downarrow \rightarrow \Phi_e(\hat{F} \oplus A) \restrict n \in T )).
\]
Because the set of conditions $(F,C)$ which force $\Phi_e(G \oplus A) \not \in [T]$ is a $\Sigma_3^A$ set, we know that for any 
3-$A$-generic $G$, there is a condition $(F,C)$ such that $G$ extends $(F,C)$ and either $(F,C)$ forces $\Phi_e(G \oplus A) 
\not \in [T]$ or $(F,C)$ has no extension that forces $\Phi_e(G \oplus A) \not \in [T]$.  In other words, $\Phi_e(G \oplus A)$ is 
either forced into or out of $[T]$.  

For a contradiction, suppose that $\Phi_e(G \oplus A)$ is a path in $[T]$.  There is a condition $(F,C)$ which is extended by $G$ 
and which forces $\Phi_e(G \oplus A)$ to be total and $\Phi_e(G \oplus A) \in [T]$.  Because $(F,C)$ forces 
$\Phi_e(G \oplus A)$ is total, there is a set $\hat{G} \leq_T A$ such that $\hat{G}$ extends $(F,C)$ and $\Phi_e(\hat{G} \oplus A)$ 
is total.  Furthermore, because $(F,C)$ forces $\Phi_e(G \oplus A) \in [T]$, each initial segment $\Phi_e(\hat{G} \oplus A) \restrict n$ 
must be an element of $T$.  Therefore, $\Phi_e(\hat{G} \oplus A) \in [T]$.  However, $\hat{G} \oplus A \leq_T A$, 
so we have a contradiction to the fact that $A$ does not compute a path in $T$.  
\end{proof}

\begin{defn}
For any set $X$, the \textbf{principal function} $p_X$ is defined by $p_X(n) =$ the $(n+1)^{\text{st}}$ element of $X$.  
\end{defn}

\begin{lem}\label{dom}
Let $G$ be weakly 3-$A$-generic for forcing with $A$-computable Mathias conditions. The principal function $p_G$ of $G$ 
dominates all functions recursive in $A$.
\end{lem}

\begin{proof}
Let $e$ be any index for which $\Phi_e(A)$ is total.  For any condition $(F,C)$, we can $A$ computably thin out $C$ to $C' \subseteq C$ 
such that $p_{F \cup C'}$ dominates $\Phi_e(A)$.  Furthermore, $(F,C')$ will be an extension of $(F,C)$.  Therefore, the set of 
conditions $(\hat{F},\hat{C})$ for which $p_{\hat{F} \cup \hat{C}}$ dominates $\Phi_e(A)$ is dense and is also a $\Sigma_3^A$ set.  
Therefore, $G$ must meet each such set of conditions.  
\end{proof}

\begin{cor}
\label{cor:high}
If $G$ is weakly 3-$A$-generic for forcing with $A$-computable Mathias conditions, then $A'' \leq_T (G \oplus A)'$.
\end{cor}

\begin{proof}
This corollary follows immediately from Lemmas \ref{marel} and \ref{dom}.  
\end{proof}

\begin{prop}
\label{prop:nocomp}
For any infinite computable tree $T \subseteq 2^{< \omega}$ with no computable paths, there is a high-above ideal $\mathcal{I}$ such that no element of 
$\mathcal{I}$ can compute a path through $T$.  
\end{prop}

\begin{proof}
We define a sequence of sets $I_0 <_T I_1 <_T  \cdots$ such that $I_n$ does not compute a path through $T$ and 
$I_n'' \leq_T I_{n+1}'$.  $\mathcal{I} = \{ X \, | \, \exists n (X \leq_T I_n) \}$ has the 
required properties.  

Let $I_0 = \emptyset$ and notice that $I_0$ does not compute a path through $T$.  Assume that $I_n$ has been defined and does 
not compute a path through $T$.  Let $\hat{I}_n$ be a 3-$I_n$-generic with respect to computable $I_n$ Mathias forcing and let 
$I_{n+1} = I_n \oplus \hat{I}_n$.  By Lemma \ref{lem:tree}, $I_{n+1}$ does not compute a path through $T$ and by Corollary 
\ref{cor:high}, $I_n'' \leq_T I_{n+1}'$.
\end{proof} 

\begin{prop} 
\label{prop:noae}
There is a high-above ideal that includes no a.e.~dominating set.
\end{prop}

\begin{proof}
We define a chain $Q_0 \leq_T Q_1 \leq_T \cdots$ and let $\mathcal{I} = 
\{ X \, | \, (\exists n)( A \leq_T Q_n)\}$. To ensure $Q_n \leq_T Q_{n+1}$ we
define first a sequence $R_0, R_1, \ldots$ and let $Q_0 = R_0$,
$Q_{n+1}=R_{n+1}\oplus Q_n$.

Let $R_0=\Omega^{0'}$, with respect to an arbitrary universal
prefix-free machine. Let $R_{n+1}=\Omega^{Q_n'}$.  
Note that $R_0$ is 2-random and each $R_{n+1}$ is 2-random
relative to $Q_n$. Hence by van Lambalgen's Theorem, each $R_n$
and $Q_n$ is 2-random. Furthermore, Kautz \cite{kautz:phd} proved that 2-randoms are GL$_1$, so we have that 
$Q_n' \equiv_T Q_n \oplus 0'$ for all $n$.  Using this fact and Property 3 of Lemma \ref{lem:omegaprop}, it follows that 
$$Q_{n+1}'  \equiv_T Q_{n+1} \oplus 0' = R_{n+1} \oplus Q_n \oplus 0' \equiv_T R_{n+1} \oplus Q_n' \equiv_T Q_n''.$$

Since each $Q_n$ is 2-random (and hence not a.e.~dominating by Theorem \ref{2randnotAED}), the ideal generated by the $Q_n$'s
contains no a.e.~dominating set.
\end{proof}

\end{document}